\RequirePackage{amsthm}

\documentclass[ pdflatex, sn-apa]{sn-jnl}

\usepackage{graphicx}%
\usepackage{multirow}%
\usepackage{amsmath,amssymb,amsfonts}%
\usepackage{amsthm}%
\usepackage{mathrsfs}%
\usepackage[title]{appendix}%
\usepackage{xcolor}%
\usepackage{textcomp}%
\usepackage{manyfoot}%
\usepackage{booktabs}%
\usepackage{algorithm}%
\usepackage{algorithmicx}%
\usepackage{algpseudocode}%
\usepackage{listings}%
\usepackage{makecell}

\theoremstyle{thmstyleone}%
\newtheorem{theorem}{Theorem}
\theoremstyle{thmstyletwo}%

\theoremstyle{thmstylethree}%

\makeatletter
\renewcommand\@biblabel[1]{#1.}
\makeatother

\begin{document}

\title[Computational and Algebraic Structure of Board Games]{Computational and Algebraic Structure of Board Games}


\author*[1]{\sur{Chun-Kai} \fnm{Hwang}}\email{chunkai.hwang@gmail.com}

\author[1]{\sur{John Reuben} \fnm{Gilbert}}\email{johngilbert1182@gmail.com}

\author[2,3]{\sur{Tsung-Ren} \fnm{Huang}}\email{trhuang@ntu.edu.tw}

\author[4]{\sur{Chen-An} \fnm{Tsai}}\email{catsai@ntu.edu.tw}

\author*[1]{\sur{Yen-Jen} \fnm{Oyang}}\email{yjoyang@csie.ntu.edu.tw}


\affil[1]{\orgdiv{Department of Computer Science and Information Engineering, College of Electrical Engineering and Computer Science}, \orgname{National Taiwan University}, \orgaddress{\city{Taipei}, \postcode{106}, \state{Taiwan}, \country{R.O.C}}}

\affil[2]{\orgdiv{Department of Psychology}, \orgname{National Taiwan University}, \orgaddress{\city{Taipei}, \postcode{106}, \state{Taiwan}, \country{R.O.C}}}

\affil[3]{\orgdiv{Institute of Applied Mathematical Sciences}, \orgname{National Taiwan University}, \orgaddress{\city{Taipei}, \postcode{106}, \state{Taiwan}, \country{R.O.C}}}

\affil[4]{\orgdiv{Department of Agronomy}, \orgname{National Taiwan University}, \orgaddress{ \city{Taipei}, \postcode{106}, \state{Taiwan}, \country{R.O.C}}}


\abstract{We provide two methodologies in the area of computation theory to solve optimal strategies for board games such as Xi Gua Qi and Go. From experimental results, we find
relevance to graph theory, matrix representation, and mathematical consciousness. We prove that the decision strategy of movement for Xi Gua Qi and Chinese checker games
belongs to a subset that is neither a ring nor a group over set Y=\{-1,0,1\}. Additionally, the movement for any
board game with two players belongs to a subset that is neither a ring nor a group from the razor of Occam. We derive the closed form of the transition matrix for any board game with two players such as chess and Chinese chess. We discover that the element of the transition matrix belongs to a rational number. We propose a different methodology based on algebra theory to analyze the complexity of board games in their entirety, instead of being limited solely to endgame results. It is probable that similar decision processes of people may also belong to a matrix representation that is neither a ring nor a group.}




\maketitle

\section{Introduction}\label{sec1}

Pebble games (\citeauthor{kfoury1997infinite}, \citeyear{kfoury1997infinite}; \citeauthor{lee2008pebble}, \citeyear{lee2008pebble}) and relational games (\citeauthor{venema1990relational}, \citeyear{venema1990relational}) are first order logic (\citeauthor{kfoury1997infinite}, \citeyear{kfoury1997infinite}; \citeauthor{halpern2008using}, \citeyear{halpern2008using}). In this paper, we use a traditional board game, Xi Gua Qi, to discuss the phenomena about mathematics and computation theory. The playing rules of Xi Gua Qi are similar to the game Go, such that when an opponent's pieces are blocked, they become captured, and the endgame result has three possible states of a win, loss, or draw.

\par
Go and Xi Gua Qi are tri-valued logic, differing from Chinese checkers which is fourth-valued logic. We proposed two methods that can be applied to solve games such as Xi Gua Qi and Go.

\par
In addition, we demonstrate that the movement of pieces for any board game with two players such as Chinese chess, chess, or Go belongs to a subset of algebra structure that is neither a ring nor a group, in which the transition matrix is only with different dimensions of complexity over
\begin{math}\mathbb{Q}\end{math}, the rational numbers. We apply matrix representation theory in abstract algebra to analyze the complexity of board games as a novel methodology to analyze the complexity of whole board games.

\par
In graph theory, graphs can be categorized as being directed acyclic graphs (DAG) and cyclic graphs (CG). We found that Xi Gua Qi can be represented as either a DAG or CG. 

\par
A graph $G$ can be represented as a tuple of its set of vertices $V$ and edges $E$. The simple graph topology model of the initial playing state of Xi Gua Qi in Figure~\ref{fig1} can be presented as $G=(V,E)$, where
\begin{equation}
V=\{x \mid 0 \leq x \leq 20,x \in N\},
\end{equation}
\begin{equation}
\begin{split}
E=\{(0,1),(0,2),(0,3),(0,4),(1,2),(1,0),(1,4),(1,5),(2,0),(2,1),(2,3),(2,6),\\
    (3,0),(3,2),(3,7),(3,4),(4,0),(4,1),(4,3),(4,8),(5,1),(5,10),(5,9),(5,20),(6,2),\\
    (6,11),(6,12),(6,13),(7,3),(7,14),(7,15),(7,16),(8,4),(8,17),(8,18),(8,19),\\
    (9,5),(9,10),(9,20),(10,5),(10,9),(10,11),(11,6),(11,12),(11,10),(12,6),\\
    (12,13),(12,11),(13,6),(13,12),(13,14),(14,7),(14,13),(14,15),(15,7),(15,14),\\
    (15,16),(16,7),(16,15),(16,17),(17,8),(17,16),(17,18),(18,8),(18,17),(18,19),\\
    (19,8),(19,18),(19,20),(20,5),(20,9),(20,19)\}.
\end{split}
\end{equation}
\par

\begin{figure}[tb]
\includegraphics[scale=1.0]{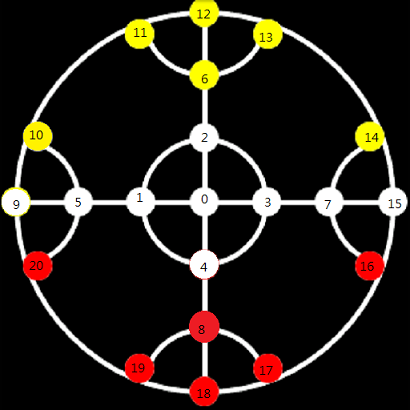}
\caption{The initial playing state of Xi Gua Qi}
\label{fig1}
\end{figure}

Here, we consider the whole game as a higher ordered complex graph. If we view each snapshot of the game board as a vector point, it can be represented as the vertex of a graph. Each snapshot of the game board is \begin{math}a \in A=X^{21}\end{math}, where \begin{math} X=\{1,2,3\}\end{math}, in which the indices 1 and 2 represent the two possible colors of game pieces, and 3 represents empty. Every game movement is represented by an edge of the graph, thereby allowing the whole game to be composed of a graph. Each edge is a functional operator (a matrix) or a function mapping \begin{math}f \end{math}, \begin{math}f:A \mapsto A\end{math}, where \begin{math} A=X^{21}, X=\{1,2,3\} \end{math}.
\par
In computational spin networks (chapter 2 of \citeauthor{moffat2020mathematically}, \citeyear{moffat2020mathematically}), every eigenstate can be represented as a vector, and the matrix operator is the edge for dual Hilbert space ( \citeauthor{moffat2020mathematically}, \citeyear{moffat2020mathematically}; \citeauthor{sieradski1992introduction}, \citeyear{sieradski1992introduction}). We apply this concept to the high ordered whole game graph. From Theorem~\ref{thm3}, we found that the game movement is a function mapping $f$ that belongs to \begin{math} D=\mathbf{M}_n(\mathbb{Q})\setminus\{\mathbf{I}_{n \times n},\mathbf{0}_{n \times n},\mathbf{Z}_{n \times n}\}\end{math}. \begin{math}\mathbf{M}_n(\mathbb{Q})\end{math} is the set of $n \times n$ matrices over $\mathbb{Q}$, the rational numbers. D is the set \begin{math}\mathbf{M}_n(\mathbb{Q})\end{math} that excludes identity matrix, zero matrix, and a matrix that contains at least one zero row vector. Additionally, \begin{math}\mathbf{M}_n (\mathbb{Q})\end{math} is a noncommutative ring and an abelian group under addition. D is neither a ring nor a group.

\section{Related Works}\label{sec2}
Table \ref{tab1} is a summary of different approaches to analyse chess games. \citeauthor{morse1944unending} (\citeyear{morse1944unending})  discussed the symbolic dynamics and the semigroups based on the unending chess. \citeauthor{stiller1996multilinear} (\citeyear{stiller1996multilinear}) applied multilinear algebra to construct the model to analyze chess endgames ( \citeauthor{bourzutschky2005chess}, \citeyear{bourzutschky2005chess}; \citeauthor{nalimov2000space}, \citeyear{nalimov2000space}) from humans and computers. Chess endgames can be analyzed with Kronecker tensor product and direct sums. They discussed the work of Friedrich Amelung and Theodor Molien, the founder of group representation, and the first person that analyzed a pawnless endgame (\citeauthor{stiller1995exploiting}, \citeyear{stiller1995exploiting}; \citeauthor{good1979p}, \citeyear{good1979p}). High-performance parallel computing can be solved from the endgame analysis.

\begin{table}[htbp]
\begin{minipage}{\textwidth}
\caption{Comparison of game strategy analysis approaches}
\label{tab1}
	\centering
		\begin{tabular}{c c c}
			\hline
			{Presenter}&{Concept} &{Game type}\\
			\hline              
{Ours}&{Graph, Group representation}&{A whole game}\\
{\citeauthor{morse1944unending} 
(\citeyear{morse1944unending})}&
{Semigroups}& {Chess endgames}\\
{\citeauthor{stiller1996multilinear} 
(\citeyear{stiller1996multilinear})}&{Multilinear algebra}&{Chess endgames}\\
{\citeauthor{schrittwieser2020mastering} 
(\citeyear{schrittwieser2020mastering})}&{Reinforcement learning}&{Game strategy}\\
{Mehta1 et al. (\citeyear{mehta2020predicting})} &{Multilayer perceptron}&{Game strategy}\\ 
{Noever1 et al. (\citeyear{noever2020chess})} & {Natural language transformer}&{Game strategy}\\
{\citeauthor{zhang2012quantum} (\citeyear{zhang2012quantum})}&{Quantum game theory}&{Game strategy}\\

        \hline
		\end{tabular}
\end{minipage}
\end{table}

\par
Chess endgames have been studied for over a century. It has been applied to improve the prediction accuracy and efficiency of reinforcement learning, data compression (\citeauthor{gomboc2022chess}, \citeyear{gomboc2022chess}), and two-level logic minimization (\citeauthor{gomboc2022chess}, \citeyear{gomboc2022chess}). The related algorithms for the logic minimization are MINI (\citeauthor{hong1974mini}, \citeyear{hong1974mini}), ESPRESSO (\citeauthor{kanakia2021espresso}, \citeyear{kanakia2021espresso}; \citeauthor{coudert2002two}, \citeyear{coudert2002two}) and Pupik (\citeauthor{fiser2008fast}, \citeyear{fiser2008fast}). It can also be applied to electronic design automation (EDA). 

\par
\citeauthor{schrittwieser2020mastering} (\citeyear{schrittwieser2020mastering}) developed a reinforcement learning algorithm to solve strategies for Atari, Go, chess and shogi games, while requiring approximately a million training steps. They compared different agents, or algorithms, including Ape-X, R2D2, MuZero, IMPALA, Rainbow, UNREAL, LASER, MuZero Reanalyze. Among these agents, MuZero Reanalyze achieved the best performance. \citeauthor{mehta2020predicting} (\citeyear{mehta2020predicting}) predicted chess moevement by using the multilayer perceptron model. They used a chess board evaluation function that could be applied to evaluate the board without deep lookahead search algorithms. They developed a chess engine, thereby avoiding the use of state space search to find the next optimal movement. \citeauthor{noever2020chess} (\citeyear{noever2020chess}) applied natural language transformers to support more generic strategic modeling, especially for text-archived games. Their approach focused on the breadth search of millions of games that would allow a language model to define a game’s rules and strategy by itself. \citeauthor{zhang2012quantum} (\citeyear{zhang2012quantum}) presented a simple but complete model to extend the quantum strategic game theory. 

\section{Simulation: Xi Gua Qi}\label{sec3}
Xi Gua Qi is one kind of pebble game with two players. The initial state of the board is presented in Figure~\ref{fig1}, with yellow and red representing the pieces of the two players.

\subsection{Method 1:  Game Tree}\label{subsec1}
We constructed an objective function that optimized the degree of freedom (DoF), i.e., the number of game pieces that are not captured by the opponent. The optimization objective value for a player is to maximize its DoF while minimizing its opponent’s DoF.  It is a general algorithm that can be applied with a min-max game tree (\citeauthor{rivest1987game}, \citeyear{rivest1987game}), and it is applicable to both Xi Gua Qi and Go. In order to determine that the opponent’s game pieces are captured, we can derive a mutable tree presented in Figure~\ref{fig3}. If the leaf nodes are comprised entirely of a single player's colors, then the opponent’s game pieces have all been captured.

\begin{figure}[tb]
\includegraphics[scale=1.0]{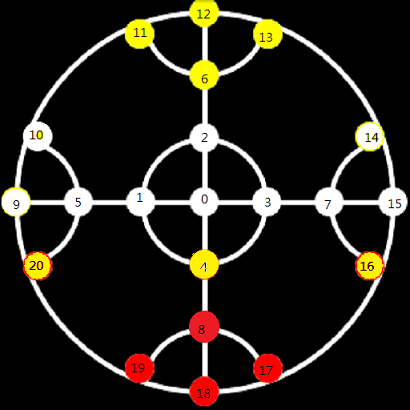}
\caption{The Graph of the board during playing state}
\label{fig2}
\end{figure}

\begin{figure}[tb]
\includegraphics[scale=1.0]{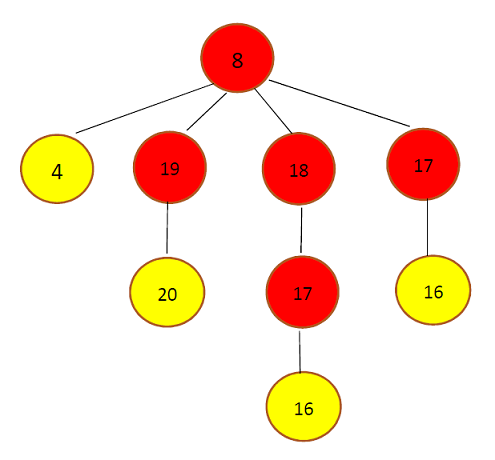}
\caption{The mutable tree for the game pieces of red color blocked by the opponent’s yellow color pieces that correspond to Figure~\ref{fig2}}
\label{fig3}
\end{figure}

\subsection{Method 2:  Probabilistic Neural Network Rule Models}\label{subsec2}

We proposed PBCR1 and PBCR2 algorithms (\citeauthor{hwang2023extracting}, \citeyear{hwang2023extracting}) to extract probabilistic Boolean rule models from Deep Neural Networks (DNN), which can also be applied to predict game movement. It basically can be applied to auto-inference (\citeauthor{hwang2023extracting}, \citeyear{hwang2023extracting}).

\par
We recorded the 100 playing games of pieces movements, and among these games, the win ratio of red color to yellow color is 50 to 50, and as a result we regarded it as a binary classification problem. There were a total of 4985 piece movements, with 2519 win (positive) cases and 2466 loss (negative) cases. 
The training features of this problem were the movement of pieces. The input was the current chess state, the next state, and the current playing color. The current state was composed of 21 nodes. Each game piece used a triple value to denote it, red, yellow, and the empty game pieces respectively. We use 2 bits to store the triple value. The current and the next state needed 42 nodes, and they were composed of 42×2=84 bits, with 1 bit also used to denote the current playing color. As a result, there were a total of 85 inputs and 1 binary output (win or loss).

\section{Results}

\begin{table}[htbp]
\begin{minipage}{\textwidth}
\caption{Performance metrics of DNN, PBCR1, PBCR2, and DT}
\label{tab2}
	\centering
        \begin{tabular}{c c c c c}
        \hline
	\makecell{Metrics\\(mean\begin{math}\pm\end{math}SD)}& {DNN}& {PBCR1}& {PBCR2} & {DT}\\
	\hline
{AUC}	   & {0.971$\pm$0.003} & {0.814$\pm$0.018} & {0.798$\pm$0.020} & {0.855$\pm$0.017}\\
{Accuracy} & {88.435\%$\pm$0.385\%} & {75.486\%$\pm$2.212\%} & {75.236\%$\pm$2.089\%} & {76.239\%$\pm$1.723\%} \\
{Sensitivity} & {80.051\%$\pm$0.046\%} & {79.866\%$\pm$0.315\%} & {79.959\%$\pm$0.368\%} & {80.172\%$\pm$0.239\%} \\
{PPV} & {95.504\%$\pm$0.910\%} & {72.715\%$\pm$3.057\%} & {72.324\%$\pm$2.856\%} & {74.164\%$\pm$2.341\%} \\
{NPV} & {83.556\%$\pm$0.107\%} & {78.790\%$\pm$1.010\%} & {78.736\%$\pm$0.957\%} & {78.700\%$\pm$0.846\%} \\
{Specificity} & {96.409\%$\pm$0.757\%} & {71.321\%$\pm$4.367\%} & {70.744\%$\pm$4.123\%} & {72.361\%$\pm$3.381\%} \\
{Odds ratio} & {113.404$\pm$29.428} & {10.187$\pm$2.243} & {9.928$\pm$2.145} & {10.796$\pm$1.806} \\
{F1} & {0.871$\pm$0.004} & {0.761$\pm$0.016} & {0.759$\pm$0.016} & {0.770$\pm$0.013} \\
			\hline
		\end{tabular}
\end{minipage}
\end{table}

Table \ref{tab2} presents the performance metrics of DNN, PBCR1, PBCR2 and decision
tree (DT) methods. In general, the performance of the DNN on each metric was better than DT, PBCR1, and PBCR2. The comparison was based on the same sensitivity on training fold
to be 80\% statistically equivalent, such that the confidence intervals of the sensitivity with different algorithms were overlapped.

\section{Discussion}
We found some interesting phenomena during different kinds of board games. If we regard each snapshot of the board game as a vector point, it can be regarded as the vertex of a graph. Every piece movement is the edge of the graph, 
The edge is a function mapping or the transition matrix  \begin{math}f:A \mapsto A\end{math}, where \begin{math}A = X^{21}, X=\{1,2,3\}\end{math}, in which 1 and 2 represent the different colors of the chess pieces, and 3 represents the empty, such that the whole game composed a graph. The win or loss can be either a directed acyclic graph (DAG) or a cyclic graph (CG), such that if there were no duplicated vertices of the graph, it would be a DAG, and a CG otherwise. 
However, the drawn game is a CG. The cardinality of the vertex would be infinite but with a finite element of the vertex set. Therefore, the graph composed from a drawn game must contain duplicate vertices. As a result, it is a CG.

\begin{theorem}[Neither A Ring nor A Group]\label{thm1}
If we view the snapshot of the board game as a vertex of a graph, the edge of this graph represents a piece movement, and the edge is an operator (a matrix) or a function mapping f, \begin{math}f:A \mapsto A\end{math}, where 
\begin{math}A=X^{21}, X=\{1,2,3\}\end{math}, 1, and 2 are the different colors of the chess pieces, and 3 represents the empty, then the following properties hold:
\begin{enumerate}
\item There exists a function mapping $f$ belonging to a subset of \begin{math}D=\mathbf{M}_{21}(Y)\setminus\{\mathbf{I}_{21 \times 21}, \mathbf{0}_{21 \times 21}, \mathbf{Z}_{21 \times 21}\}\end{math}, where \begin{math}\mathbf{M}_{21}(Y)\end{math} is the set of $21 \times 21$ matrices over $Y$, \begin{math}Y=\{-1,0,1\}\end{math}, \begin{math}\mathbf{I}_{21 \times 21}\end{math} is the identity matrix, \begin{math}\mathbf{0}_{21 \times 21}\end{math} is the zero matrix, and \begin{math}\mathbf{Z}_{21 \times 21}\end{math} is a matrix that contains at least one zero row vector.   
\item 
\begin{math}\mathbf{M}_{21}(Y)\end{math} is a noncommutative ring in abstract algebra.
\item
\begin{math}\mathbf{M}_{21}(Y)\end{math} is an abelian group under addition.
\item
\begin{math}D=\mathbf{M}_{21}(Y)\setminus\{\mathbf{I}_{21 \times 21}, \mathbf{0}_{21 \times 21}, \mathbf{Z}_{21 \times 21}\}\end{math} is neither a ring nor a group.
\end{enumerate}
\end{theorem}

\begin{proof}
\begin{enumerate}
    \item
For a vertex \begin{math}\mathbf{a}=[a_i]\end{math}, where \begin{math}i=0,…,20\end{math}, if we move the game piece from \begin{math}p\end{math} to \begin{math}q\end{math}, then either \begin{math}a_p=1\end{math} or \begin{math}a_p=2\end{math}. Furthermore, \begin{math}a_q=3\end{math} as \begin{math}q\end{math} is empty. The function mapping \begin{math}f\end{math} is a \begin{math}21 \times 21\end{math} matrix \begin{math}\mathbf{M}_{21}\end{math}. 

\par
First, consider the case that \begin{math}a_p=1\end{math}. Suppose that after the game movement from $p$ to $q$, we will get a vertex \begin{math}\mathbf{b}=[b_i]=\mathbf{M}_{21} \mathbf{a}\end{math}, with the properties that for game pieces \begin{math}b_p=3\end{math}  and \begin{math}b_q=1\end{math}. Pieces with indices \begin{math}k \in K\end{math} are captured with this movement, resulting in \begin{math}b_k=3\end{math}. Similarly, if the movement does not cause the game pieces to be captured, then \begin{math}b_i=a_i\end{math} for \begin{math}\forall i \neq p, q\end{math}, and \begin{math}i \notin K\end{math}. 

\par
Next, we show how to construct the matrix \begin{math}\mathbf{M}_{21}=[\sigma_{ij}]\end{math} with Occam's razor principle. Suppose that the game piece $s$ belongs to the opponent, i.e., \begin{math}a_s=2\end{math}, for the $p$-th row vector \begin{math}\mathbf{r}_p\end{math} of \begin{math}\mathbf{M}_{21}\end{math}, we can set \begin{math}\sigma_{pp}=1, \sigma_{ps}=1\end{math}, and \begin{math}\sigma_{pj}=0, \forall j \neq p,s\end{math}. Thus, we can construct the transformation that \begin{math}b_p=3\end{math}. 

\par
For the $q$-th row vector \begin{math}\mathbf{r}_q\end{math} of \begin{math}\mathbf{M}_{21}\end{math}, we can set  \begin{math}\sigma_{qq}=1,\sigma_{qs}=-1\end{math}, and \begin{math} \sigma_{qj}=0, \forall q \neq p,s\end{math}. Hence, we can construct the transformation where \begin{math}b_q=1\end{math}. 

\par
For the pieces \begin{math} k \in K\end{math} that are captured, which can be obtained from the tree from Method 1, the $k$-th row vector \begin{math}\mathbf{r}_k\end{math} of \begin{math}\mathbf{M}_{21}\end{math} is of the form \begin{math}\sigma_{kk}=1, \sigma_{kp}=1\end{math}, and  \begin{math}\sigma_{kj}=0,\forall j \neq k\end{math}. Therefore, we can construct the transformation in which  \begin{math}b_k=3\end{math}.

\par
For the pieces $i$ that are not captured and do not equal $p, q$, and do not belong to $K$, the $i$-th row vector \begin{math}\mathbf{r}_i\end{math} of \begin{math}\mathbf{M}_{21}\end{math} is of the form \begin{math}\sigma_{ii}=1,\sigma_{ij}=0, \forall j \neq i, \forall i \neq p, q\end{math}, and \begin{math}i \notin K\end{math}. 

\par
Hence, we construct the transformation that \begin{math}b_i=a_i\end{math} for \begin{math}\forall i \neq p, q\end{math}, and \begin{math}i \notin K\end{math}.
Similarly, we can construct the matrix \begin{math}\mathbf{M}_{21}\end{math} for the case that \begin{math}a_p=2\end{math}. Finally, we observed that the element of the matrix \begin{math}\mathbf{M}_{21}\end{math}, \begin{math}\sigma_{ij} \in \{-1,0,1\}\end{math}.
\par
The board vertex will be the same with omitting a movement by the transformation of an identity matrix, \begin{math}\mathbf{I}_{21 \times 21}\end{math}. The board vertex will be filled with the same undefined kinds of pieces by the transformation of a zero matrix, \begin{math}\mathbf{0}_{21 \times 21}\end{math}. If the \begin{math}\mathbf{Z}_{21 \times 21}\end{math} matrix multiplies the board vertex column vector, there will be undefined pieces by the transformation of a zero row vector. Therefore, the set of game movement matrix excludes \begin{math}\mathbf{I}_{21 \times 21}\end{math}, \begin{math}\mathbf{0}_{21 \times 21}\end{math}, and \begin{math}\mathbf{Z}_{21 \times 21}\end{math}.

    \item
Since \begin{math}Y=\{-1,0,1\}\end{math} is isomorphic to the set \begin{math}\mathbb{Z}_3=\{0,1,2\}\end{math}. The set \begin{math}\mathbb{Z}_3\end{math} is a ring, the operator $f$ is equivalent to a $21 \times 21$ matrix defined on \begin{math}\mathbb{Z}_3\end{math}. \begin{math}\mathbf{M}_{21} (\mathbb{Z}_3)\end{math} of all $21 \times 21$ matrices over \begin{math}\mathbb{Z}_3\end{math} is a noncommutative ring with identity \begin{math}\mathbf{I}_n\end{math} from Theorem F.1 (\citeauthor{hungerford2012abstract}, \citeyear{hungerford2012abstract}). It is stated as follows, “If $R$ is a ring with identity, then the set \begin{math}M_n (R)\end{math} of all $n \times n$ matrices over $R$ is a noncommutative ring with identity \begin{math}I_n\end{math}”. Thus, \begin{math}\mathbf{M}_{21}(Y)\end{math} is a noncommutative ring.
    \item 
From Theorem 7.1 (\citeauthor{hungerford2012abstract}, \citeyear{hungerford2012abstract}), every ring is an abelian group under addition and the statement of (2) that \begin{math}\mathbf{M}_{21}(Y)\end{math} is a noncommutative ring, it is trivial that \begin{math}\mathbf{M}_{21}(Y)\end{math} is an abelian group under addition.  
    \item 
 Since the set of game piece movement matrix excludes \begin{math}\{\mathbf{I}_{21 \times 21}, \mathbf{0}_{21 \times 21}\end{math}, \begin{math}\mathbf{Z}_{21 \times 21} \}\end{math}. This set lacks the identity element. Therefore, the game piece movement set is neither a ring nor a group.
\end{enumerate}
\end{proof} 

\begin{theorem}[Neither A Ring nor A Group for Chinese Checker Game]\label{thm2}
If we view the snapshot of the game board as a vertex of a graph, in which the edge of this graph is a piece movement, the edge is an operator (a matrix) or a function mapping $f$, \begin{math}f:A \mapsto A\end{math}, where \begin{math}A=X^n\end{math}, \begin{math}X=\{1,2,3,4\}\end{math}, 1,2,3 denote different categories of game pieces, and 4 denotes the empty. Additionally, $n$ is the number of the game board lattices, then the following properties hold:
\begin{enumerate}
    \item
There exists a function mapping $f$ belonging to a subset of
\begin{math}D=\mathbf{M}_n(Y)\setminus\{\mathbf{I}_{n \times n}, \mathbf{0}_{n \times n}, \mathbf{Z}_{n \times n}\}\end{math}, where \begin{math}\mathbf{M}_n(Y)\end{math} is the set of $n \times n$ matrices over $Y$, \begin{math}Y=\{-1,0,1\}\end{math}, \begin{math}\mathbf{I}_{n \times n}\end{math} is the identity matrix, \begin{math}\mathbf{0}_{n \times n}\end{math} is the zero matrix, and \begin{math}\mathbf{Z}_{n \times n}\end{math} is a matrix that contains at least one zero row vector.
    \item
\begin{math}\mathbf{M}_n(Y)\end{math} is a noncommutative ring in abstract algebra.
    \item
\begin{math}\mathbf{M}_n(Y)\end{math} is an abelian group under addition.
    \item
\begin{math}D=\mathbf{M}_n(Y)\setminus\{\mathbf{I}_{n \times n}, \mathbf{0}_{n \times n}, \mathbf{Z}_{n \times n}\}\end{math} is neither a ring nor a group.
\end{enumerate}

\end{theorem}
\begin{proof}

For a vertex \begin{math}\mathbf{a}=[a_i]\end{math}, where $i=0,\ldots ,n$. If we move the game piece from $p$ to $q$, then either \begin{math}a_p=1, 2, or 3\end{math}. 
Additionally, \begin{math}a_q=4\end{math} since $q$ is empty. The function mapping $f$ is an $n \times n$ matrix \begin{math}\mathbf{M}_n\end{math}. 

\par
First, we consider the case that \begin{math}a_p=1\end{math}. Suppose that after the game movement from $p$ to $q$, we get a vertex \begin{math}\mathbf{b}=[b_i]=\mathbf{M}_n \mathbf{a}\end{math}, with the properties that for game pieces \begin{math}b_p=4\end{math} and \begin{math}b_q=1\end{math}. For other lattices on the game board, we would have \begin{math}b_i=a_i\end{math} for  \begin{math}\forall i \neq p, q\end{math}.


\par
Now, we show how to construct the matrix \begin{math}\mathbf{M}_n=[\sigma_{ij}]\end{math} with Occam's razor principle. Suppose that the game piece $s$ belongs to other opponents, i.e., \begin{math}a_s=3\end{math}, for the $p$-th row vector \begin{math}\mathbf{r}_p\end{math} of \begin{math}\mathbf{M}_n\end{math}, we can set \begin{math}\sigma_{pp}=1, \sigma_{ps}=1\end{math}, and \begin{math}\sigma_{pj}=0, \forall j \neq p, s\end{math}. Hence, we construct the transformation that \begin{math}b_p=4\end{math}. 

\par
For the $q$-th row vector \begin{math}\mathbf{r}_q\end{math} of \begin{math}\mathbf{M}_n\end{math}, we can set \begin{math}\sigma_{qq}=1, \sigma_{qs}=-1\end{math}, and \begin{math}\sigma_{qj}=0, \forall q \neq p, s\end{math}. Thus, we construct the transformation that \begin{math}b_q=1\end{math}. 

\par
For other lattices $i$ that does not equal $p$, $q$, the $i$-th row vector \begin{math}\mathbf{r}_i\end{math} of \begin{math}\mathbf{M}_n\end{math} is of the form \begin{math}\sigma_{ii}=1, \sigma_{ij}=0, \forall j \neq i, \forall i \neq p, q\end{math}. Hence, we construct the transformation that \begin{math}b_i=a_i\end{math} for \begin{math}\forall i \neq p, q\end{math}.

\par
Similarly, we can construct the matrix \begin{math}\mathbf{M}_n\end{math} for the case that \begin{math}a_p=2,3\end{math}. Finally, we observed that the element of the matrix \begin{math}\mathbf{M}_n, \sigma_{ij} \in \{-1,0,1\}\end{math}.
\par
The board vertex will be the same with omitting a movement by the transformation of an identity matrix, \begin{math}\mathbf{I}_{n \times n}\end{math}. The board vertex will be filled with the same undefined kinds of pieces by the transformation of a zero matrix, \begin{math}\mathbf{0}_{n \times n}\end{math}. If a zero row vector multiplies the board vertex column vector, there will be an undefined piece. Therefore, the set of chess movement matrix excludes \begin{math}\mathbf{I}_{n \times n}\end{math}, \begin{math}\mathbf{0}_{n \times n}\end{math}, and \begin{math}\mathbf{Z}_{n \times n}\end{math}.
\par
(2), (3), and (4) are proved in a similar procedure as Theorem~\ref{thm1}.

\end{proof}

\begin{theorem}[Neither A Ring nor A Group for Any Board Game with Two Players such as Chess and Chinese Chess]\label{thm3}
If we view the snapshot of the game board as a vertex of a graph, in which the edges of the graph are game piece movements, the edge is an operator (a matrix) or a function mapping $f$, \begin{math}f:A \mapsto A\end{math}, where \begin{math}A=X^n, X=\{1,2,3,..t,t+1,...,d,d+1\}\end{math}, $1,2,3,\ldots,t$, denotes $t$ different categories of game pieces for player Alice, $t+1,\ldots,d$ denotes $(d-t)$ different categories of game pieces for player Bob, and $d+1$ denotes the empty. Additionally, $n$ is the number of the board game lattices, then the following properties hold:
\begin{enumerate}
    \item
There exists a function mapping $f$ that belongs to a subset of \begin{math}D=\mathbf{M}_n(\mathbb{Q})\setminus\{\mathbf{I}_{n \times n}, \mathbf{0}_{n \times n}, \mathbf{Z}_{n \times n}\}\end{math}, where \begin{math}\mathbf{M}_n(\mathbb{Q})\end{math} is the set of $n \times n$ matrices over $\mathbb{Q}$, the rational numbers, \begin{math}\mathbf{I}_{n \times n}\end{math} is the identity matrix, \begin{math}\mathbf{0}_{n \times n}\end{math} is the zero matrix, and \begin{math}\mathbf{Z}_{n \times n}\end{math} is the matrix that contains at least one zero row vector. 
    \item
\begin{math}\mathbf{M}_n(\mathbb{Q})\end{math} is a noncommutative ring in abstract algebra.
    \item
\begin{math}\mathbf{M}_n(\mathbb{Q})\end{math} is an abelian group under addition.
    \item
\begin{math}D=\mathbf{M}_n(\mathbb{Q})\setminus\{\mathbf{I}_{n \times n}, \mathbf{0}_{n \times n}, \mathbf{Z}_{n \times n}\}\end{math} is neither a ring nor a group.
\end{enumerate}
\end{theorem}

\begin{proof}
\begin{enumerate}
    \item 
Similar to the proof of Theorem~\ref{thm1}, since any \begin{math}x \in X=\{1,2,3,...,d+1\}\end{math} can be composed from an inner product of row vector \begin{math}\mathbf{y} \in \mathbb{Q}^n\end{math} and any vector \begin{math}\mathbf{a} \in X^n\end{math}, we can  construct an $n \times n$ matrix over $\mathbb{Q}$, \begin{math}\mathbf{M}_n \in \mathbf{M}_n(\mathbb{Q})\end{math}, such that 
\begin{math}\mathbf{b}=\mathbf{M}_n \mathbf{a} \in X^n\end{math}, i.e., \begin{math}\mathbf{M}_n\end{math} is composed of different row vector \begin{math}\mathbf{y} \in \mathbb{Q}^n\end{math}.

\par
For a vertex \begin{math}\mathbf{a}=[a_i]\end{math}, where \begin{math}i=0,…,n\end{math}. If we move the game piece from $p$ to $q$, then we have \begin{math}a_p=1,2,…,t,,t+1,…,d\end{math}. In addition, \begin{math}a_q=d+1\end{math} as \begin{math}q\end{math} is empty. The function mapping \begin{math}f\end{math} is an \begin{math}n \times n\end{math} matrix \begin{math}\mathbf{M}_n\end{math}. 
\par
First, we consider the case that \begin{math}
a_p=x \in X\end{math}. Suppose that after the game piece movement from $p$ to $q$, we will get a vertex \begin{math}\mathbf{b}=[b_i]=\mathbf{M}_n \mathbf{a}\end{math}, with the properties that for game pieces \begin{math}b_p=d+1\end{math} and \begin{math}b_q=x\end{math}. Pieces with indices \begin{math}k \in K\end{math} are captured with this movement, resulting in \begin{math}b_k=d+1\end{math}. In addition, if the movement does not cause the game pieces to be captured, then \begin{math}b_i=a_i\end{math} for \begin{math}\forall i \neq p, q\end{math}, and \begin{math}i \notin K\end{math}.

\par
Here, we give an intuitive simple construction of \begin{math}\mathbf{M}_n=[\sigma_{ij}]\end{math} with Occam's razor principle. Suppose that the opponent chess piece s with the properties, \begin{math}a_s=u, t+1 \leq u \leq d\end{math}. We want to derive the \begin{math}\sigma_{ps}\end{math} such that 
\begin{math}\sigma_{pp} x+\sigma_{ps}u=d+1\end{math}. Therefore, if \begin{math}\sigma_{pp}=1\end{math}, we have 
\begin{math}\sigma_{ps}=(1/u)(d+1-x) \in \mathbb{Q}\end{math}.

\par
For the $p$-th row vector \begin{math}\mathbf{r}_p\end{math} of \begin{math}\mathbf{M}_n\end{math}, we can set \begin{math}\sigma_{pp}=1\end{math}, and 
\begin{math}\sigma_{ps}=(1/u)(d+1-x) \in \mathbb{Q}\end{math}, and 
\begin{math}\sigma_{pj}=0, \forall j \neq p, s\end{math}. Hence, we construct the transformation that \begin{math}b_p=d+1\end{math}.
\par
For the $q$-th row vector \begin{math}\mathbf{r}_q\end{math} of \begin{math}\mathbf{M}_n\end{math}, we want to derive the \begin{math}\sigma_{ps}\end{math} such that
\begin{math}\sigma_{qq}(d+1)+\sigma_{qs}u=x\end{math}. Therefore, if \begin{math} \sigma_{qq}=1\end{math}, we have 
\begin{math}\sigma_{qs}=(1/u)(x-d-1) \in \mathbb{Q} \end{math}, and \begin{math}\sigma_{qj}=0, \forall q \neq p, s\end{math}. Hence, we construct the transformation that \begin{math}b_q=x\end{math}. 

\par
For the game pieces \begin{math}k \in K\end{math} that are captured, as determined by the rules of the board game, let the captured piece be denoted with value $v$. For the $k$-th row vector \begin{math}\mathbf{r}_k\end{math} of \begin{math}\mathbf{M}_n\end{math}, we want to derive the \begin{math}\sigma_{kp}\end{math} such that \begin{math}\sigma_{kk}v+\sigma_{kp}x=d+1\end{math}. We can construct \begin{math}\mathbf{M}_n\end{math} of the form \begin{math}\sigma_{kk}=1, \sigma_{kp}=(1/x)(d+1-v) \in \mathbb{Q}\end{math}, and \begin{math}\sigma_{kj}=0, \forall j \neq k\end{math}. Therefore, we construct the transformation that \begin{math}b_k=d+1\end{math}.

\par
For the game pieces $i$ that are not captured, do not equal $p$, $q$, and do not belong to $K$, the $i$-th row vector \begin{math}\mathbf{r}_i\end{math} of \begin{math}\mathbf{M}_n\end{math} is of the form \begin{math}\sigma_{ii}=1, \sigma_{ij}=0, \forall j \neq i, \forall i \neq p, q\end{math}, and \begin{math}i \notin K\end{math}. Hence, we construct the transformation that \begin{math}b_i=a_i\end{math} for \begin{math}\forall i \neq p, q, \end{math} and \begin{math}i \notin K\end{math}.
Consequently, we observed that the element of the matrix belongs to rational numbers, that is, \begin{math}\mathbf{M}_n=[\sigma_{ij}], \sigma_{ij} \in \mathbb{Q}\end{math}.
\par
The board vertex will be the same with omitting a movement by the transformation of an identity matrix, \begin{math}\mathbf{I}_{n \times n}\end{math}. The board vertex will be filled with the same undefined kinds of pieces by the transformation of a zero matrix, \begin{math}\mathbf{0}_{n \times n}\end{math}. If the \begin{math}\mathbf{Z}_{n \times n}\end{math} matrix multiplies the board vertex column vector, there will be undefined pieces by the transformation of a zero row vector. Therefore, the set of game piece movement matrix excludes \begin{math}\mathbf{I}_{n \times n}\end{math},  \begin{math}\mathbf{0}_{n \times n}\end{math}, and \begin{math}\mathbf{Z}_{n \times n}\end{math}. 
    \item
Since \begin{math}\mathbb{Q}\end{math} is a ring with identity, from Theorem F.1 (\citeauthor{hungerford2012abstract}, \citeyear{hungerford2012abstract}). \begin{math}\mathbf{M}_n(\mathbb{Q})\end{math} is a noncommutative ring.
    \item
It is trivial and similar to the proof (3) of Theorem~\ref{thm1}.
    \item 
Since the set of game piece movement matrix excludes \begin{math}\mathbf{I}_{n \times n}\end{math}, \begin{math}\mathbf{0}_{n \times n}\end{math}, and \begin{math}\mathbf{Z}_{n \times n}\end{math}, this set lacks the identity element. Hence, the game piece movement set is neither a ring nor a group.
\end{enumerate}

\end{proof}

\section{Conclusions and Future Works}
We are able to discuss the complexity of different board games from Theorem~\ref{thm1},~\ref{thm2}, and ~\ref{thm3}. From Theorem~\ref{thm1}, we found that the movements in Xi Gua Qi belong to \begin{math}D=\mathbf{M}_{21}(Y)\setminus\{\mathbf{I}_{21 \times 21}, \mathbf{0}_{21 \times 21}, \mathbf{Z}_{21 \times 21}\}\end{math}. From Theorem~\ref{thm2}, we found that the game movements of Chinese checker belong to \begin{math}\mathbf{M}_n (Y)\setminus\{\mathbf{I}_{n \times n}, \mathbf{0}_{n \times n}, \mathbf{Z}_{n \times n}\}\end{math}, where $n$ is the number of the game board lattices. From Theorem~\ref{thm3}, we found that the movements of any board game with two players belong to \begin{math}\mathbf{M}_n(\mathbb{Q})\setminus\{\mathbf{I}_{n \times n}, \mathbf{0}_{n \times n}, \mathbf{Z}_{n \times n}\}\end{math}, where \begin{math}\mathbb{Q}\end{math}  represents the rational numbers and $n$ is the numbers of the game board lattices.

\par
In Theorem~\ref{thm1}, we proved that the set of the function mapping \begin{math}f\end{math} is neither a ring nor a group. It is an interesting topic about the algebraic structure of the operator $f$, for example, whether it is with an interesting structure similar to \begin{math}SU(2)\end{math} ( \citeauthor{moffat2020mathematically}, \citeyear{moffat2020mathematically}; \citeauthor{rovelli2015covariant}, \citeyear{rovelli2015covariant}). We found the transition matrix is composed of the element from set \begin{math}Y=\{-1,0,1\}\end{math}, and it is a sparse matrix. Chess games might be related to reductive groups (\citeauthor{minchenko2011zariski}, \citeyear{minchenko2011zariski}), compact groups (\citeauthor{varju2012random}, \citeyear{varju2012random}), Lie groups (\citeauthor{breuillard2004random}, \citeyear{breuillard2004random}; \citeauthor{versendaal2019large}, \citeyear{versendaal2019large}) and random walks on groups and random transformations (\citeauthor{furman2002random}, \citeyear{furman2002random}) in abstract algebra.

\par
In quantum theory, we know that the quantum is in essence the tensor operation on the dual space (\citeauthor{moffat2020mathematically}, \citeyear{moffat2020mathematically}). We may develop a quantum game theory method (\citeauthor{zhang2012quantum}, \citeyear{zhang2012quantum}; \citeauthor{ghosh2021quantum1}, \citeyear{ghosh2021quantum1}; \citeauthor{ghosh2021quantum2}, \citeyear{ghosh2021quantum2}; \citeauthor{ghosh2021quantum3}, \citeyear{ghosh2021quantum3}) via an approach that is similar to computational spin networks (\citeauthor{moffat2020mathematically}, \citeyear{moffat2020mathematically}). 
Every state corresponds to a vector. It is probable that the human decision-making process for other strategies not limited to playing chess can be represented as a matrix operator that neither belongs to a ring nor a group, \begin{math}\mathbf{M}_n(\mathbb{R})\setminus\{\mathbf{I}_{n \times n}, \mathbf{0}_{n \times n}, \mathbf{Z}_{n \times n}\}\end{math}, where \begin{math}\mathbb{R}\end{math} represents the real numbers.

\vspace{6pt} 




\backmatter









\bigskip\noindent\textbf{Author contributions}
conceptualization, C.K.H.; methodology, C.K.H.; software, C.K.H.; validation, C.K.H., T.R.H, C.A.T. and Y.J.O; formal analysis, C.K.H.; investigation,  C.K.H. and Y.J.O.; resources, C.K.H.; data curation, C.K.H.; writing---original draft preparation, C.K.H.; writing---review and editing, C.K.H., J.R.G., T.R.H, C.A.T, and Y.J.O; visualization, C.K.H.; supervision, Y.J.O; project administration, C.K.H.;All authors have read and agreed to the published version of the manuscript.

\bigskip\noindent\textbf {Funding}
No funding was received for conducting this study.

\bigskip\noindent\textbf{Data availability}
The datasets generated and analysed during the current study are not publicly available due the fact that they constitute an excerpt of research in progress but are available from the corresponding author on reasonable request.

\bigskip\noindent\textbf{Code availability}
The codes are not publicly available but are available from the corresponding author on reasonable request.

\section*{Declarations}



\bigskip\noindent\textbf{Conflict of Interest}
The authors declare that they have no conflicts of interest.

\bigskip\noindent\textbf{Ethics approval}
Not applicable.

\bigskip\noindent\textbf{Consent to participate}
Not applicable.

\bigskip\noindent\textbf{Consent for publication}
All authors consent.

\end{document}